\documentclass[12pt,reqno]{amsart}
\usepackage[usenames]{color}
\usepackage{amsmath,verbatim,graphicx,epstopdf,enumerate}
\newcommand{\D}{\mathrm{d}}

\newcommand{\lb}{\left(}

\newcommand{\rb}{\right)}
\newcommand{\PD}{\partial}

\newcommand{\Sb}{\mathbb{S}}

\newcommand{\Beq}{\begin{equation}}
	\newcommand{\Eeq}{\end{equation}}
\newcommand{\beq}{\begin{equation*}}
	\newcommand{\eeq}{\end{equation*}}
\newcommand{\bal}{\begin{align}}
	\newcommand{\eal}{\end{align}}

\newcommand{\bp}{\begin{prob}}
	\newcommand{\bpr}{\begin{proof}}
		\newcommand{\epr}{\end{proof}}

	\newcommand{\bel}[1]{\begin{equation}\label{#1}}
		\newcommand{\ee}{\end{equation}}

	\newtheorem{theorem}{Theorem}[section]
	\newtheorem{corollary}[theorem]{Corollary}
	
	\newtheorem{lemma}[theorem]{Lemma}
	\newtheorem{proposition}[theorem]{Proposition}

	\theoremstyle{definition}
	
	\newtheorem{remark}[theorem]{Remark}

	\newcommand{\R}{{\mathbb R}}

	\newcommand{\C}{{\mathbb C}}

	\newcommand{\be}{\begin{eqnarray}}
		\newcommand{\ben}{\begin{eqnarray*}}
			\newcommand{\en}{\end{eqnarray}}
		\newcommand{\enn}{\end{eqnarray*}}

	\definecolor{rot}{rgb}{1.000,0.000,0.000}
	\definecolor{blue}{rgb}{0.000, 0.000, 1.000}

	\title[Inverse acoustic scattering with single incoming wave]{Uniqueness to inverse acoustic scattering from coated polygonal obstacles with a single incoming wave}
	\author[Hu and Vashisth]{GuangHui Hu$^\dagger$ and Manmohan Vashisth$^{\ast}$}
	\address{$^\dagger$ Beijing Computational Science Research Center, Beijing 100193, China.
		\newline
		\indent E-mail:{\tt \ hu@csrc.ac.cn}}
	\address{$^\ast$ Beijing Computational Science Research Center, Beijing 100193, China.
		\newline
		\indent E-mail:{\tt\  mvashisth@csrc.ac.cn}}
	
\begin{document}
		%
		
		
		\maketitle
		\begin{abstract}
			It is proved that a connected polygonal obstacle coated by thin layers together with its surface impedance function can be determined uniquely from the far field pattern of a single incident plane wave.
			Our proof is based on 
			the Schwarz reflection principle for the Helmholtz equation satisfying the impedance boundary condition on a flat boundary.
		\end{abstract}
		

		%
		\section{Introduction and main results}
		Let $D\subset \R^{2}$ be a coated obstacle by a thin dielectric layer, which is embedded in an infinite homogeneous medium. In this paper, $D$ is supposed to be a bounded connected polygon such that its exterior $D^c$ is connected.
		Consider the time-harmonic acoustic scattering of a plane wave $u^{in}(x)=e^{ikx\cdot d}$ from the impenetrable scatterer $D$ modelled by the following system of equations
		\begin{align}
			&\label{Helmholtz scattering}\Delta u +k^{2}u=0\qquad \mbox{in}\quad D^c:=\R^{2}\backslash\overline{D},\\
			&\label{Totla field} u(x)=e^{ikx\cdot d}+u^{sc}(x),\quad x, d\in \R^2,\\
			&\label{Sommerfeld radiation condition} \lim_{r\rightarrow \infty}\sqrt{r}\lb\PD_r u^{sc}-ik u^{sc}\rb =0,\quad r=\lvert x\rvert,\\
			&\label{Boundary condition} \PD_\nu u +i\lambda u=0\qquad \mbox{on}\ \PD D,
		\end{align}
		where $k>0$ is the wave number, $d\in\Sb:=\{x\in\R^{2}:\ \lvert x\rvert =1\}$ is the incident direction, and $u^{sc}$ is the  scattered field.  The normal direction $\nu\in \Sb$ is supposed to be outward.
		Equation  \eqref{Sommerfeld radiation condition} is known  as the Sommerfeld radiation condition and the impedance coefficient $\lambda>0$ is supposed to be constant. It is well known (see \cite{Cakoni,Cakoni_Colton_Monk,Colton_Kress_Book})  that there  exists a unique solution $u^{sc}$ to the system of Equations \eqref{Helmholtz scattering}-\eqref{Boundary condition} such that $u^{sc}\in H^{1}_{loc}(D^c)$ if $\PD D$ is Lipschitz. Moreover, the  Sommerfeld radiation condition of $u^{sc}$ leads to the asymptotic expansion
		\begin{align}\label{Asymptotic expansion}
			u^{sc}(x)=\frac{e^{ikr}}{\sqrt{r}}\lb u^{\infty}\lb\hat{x}\rb +O\lb \frac{1}{r}\rb\rb,\qquad \hat{x}:=x/|x|,
		\end{align}
		uniformly in all directions $\hat{x}\in \Sb$,
		where the function $u^{\infty}$ defined on unit circle is known as the  far field pattern. Inverse scattering problem is to determine the shape and position of the scatterer $D$ together with the impedance coefficient $\lambda$ from  knowledge of the far field pattern $u^{\infty}$. In this paper we prove the following uniqueness theorem with a single incoming wave.
		\begin{theorem}\label{Main theorem}
			Assume that $(D_{1}, \lambda_1)$ and $(D_{2},\lambda_2)$ are two connected polygonal obstacles. Let $u_{j}^{\infty}$ for $j=1,2$, denote the far-field patterns of the scattering problem \eqref{Helmholtz scattering}-\eqref{Boundary condition} with the fixed incident direction $d\in\Sb$ when $D=D_{j}$ and $\lambda=\lambda_j$, respectively. Then
			the relation $u_{1}^{\infty}(\widehat{x})=u^{\infty}_{2}(\widehat{x})$
			over all observation directions $\widehat{x}\in\Sb$ implies that $D_1=D_2$ and $\lambda_1=\lambda_2$.
		\end{theorem}

		As a by-product of the uniqueness proof to Theorem \ref{Main theorem}, we can obtain
		\begin{theorem}\label{Th2} {Let $\lambda=\lambda(x)$ $(x\in \R^2)$ be an unknown function which can be continued to an entire function on $\C^2$ such that $\lambda\geq 0$ on $\partial D$}. Then the far-field pattern $u^{\infty}(\hat{x})$ for all $\hat{x}\in \Sb$ with one incident direction uniquely determines $\partial D$ and {$\lambda(x)|_{\partial D}$}. 
		\end{theorem}
		
		The unique identification of an obstacle from knowledge of far-field pattern goes back to Schiffer \cite{Lax} in 1967, where he proved that a sound-soft obstacle can be uniquely recovered using infinite number of plane waves with distinct directions. In \cite{CS1983}, Colton and Sleeman proved the same uniqueness with a finite number of plane waves, provided a priori information on the size of the sound-soft obstacle is available. In particular, one plane wave is enough if the underlying obstacle is sufficiently small compared to the wave length; see also \cite{Gintides, Ikehata} for related discussions on the bound. 
		Global uniqueness results with a single wave were obtained within the class of sound-soft/sound-hard/penetrable polygonal and polyhedral scatterers \cite{CY,Rondi05,LHZ06,Elshner_Yamamoto_polygonal,EY08,ElHu2018acoustic,
			ElHu2015,HSV}. In addition, sound-soft balls/disks can be uniquely determined by a single far-field pattern of a plane wave or point source wave \cite{Liu, HL14}.
		
		For coated obstacles,
		Isakov \cite[Theorem 2.4]{Isakov} proved that far-field patterns excited at multi frequencies uniquely determine the shape as well as the impedance function. Two plane waves with different directions were used in \cite{Cheng2003} to determine a convex-polygonal obstacle, which were recently relaxed to a single wave in \cite{HMY}. Determination of non-convex coated polygons was discussed early in \cite{LHZ07}. However, the proof of \cite[Theorem 2.1, Page 302]{LHZ07} contains a serious gap in using normal directions of the gap domain between two different obstacles. Hence, the unique determination of a non-convex polygon of impedance type from a single far field pattern also remains open up to now,  which is the aim of this paper.
		
		We remark that the idea of using the Schwarz reflection principle \cite{CY,Rondi05,Elshner_Yamamoto_polygonal,EY08,LHZ06} for recovering sound-soft/sound-hard polygonal obstacles cannot be carried over to the impedance case straightforwardly. This is due to the fact that the \textquoteleft point-to-point\textquoteright \ reflection principle for the Helmholtz equation subject to the Dirichlet/Neumann boundary condition on a flat surface is no longer valid under the Robin boundary condition.
		In \cite{Diaz_Ludford-Reflection principle}, Diaz and Ludford established a \textquoteleft non-point-to-point' reflection principle
		under the impedance  boundary conditions given on a subset of a hyperplane
		if the domain of the Helmholtz equation fulfills an additional geometric condition in $\R^n$ ($n\geq 2$) (see also Remark \ref{remark2.3} (ii)).
		Such kind of reflection principle turns out to be sufficient in uniquely determining  coated
		convex-polygonal obstacles with a constant surface impedance; see \cite{HMY}.
		
		Our arguments for treating non-convex polygons essentially consist of two ingredients. (i) A non-local extension formula for the two-dimensional Helmhotlz equation which applies to general connected domains with the Robin boundary condition enforcing on a flat subboundary (see Section \ref{Helmholtz reflection}). This is mostly motivated by the reflection principles for the harmonic (resp. Helmholtz) equation subject to the Robin (resp. Dirichlet and Neumann) boundary condition(s) on a real-analytic subboundary \cite{Belinskiy_Savina_Schwartz_Harmonic_Robin,
			Savina_Sternin_Shatalov_Reflection_Helmholtz,
			Savina_Helmholtz_Neumann, Savina_general_elliptic_Transaction_AMS}. (ii) A novel path argument for applying
		\textquoteleft non-point-to-point\textquoteright \ reflection principles to prove the analytical extension of wave fields in polygonal domains.
		Our method for proving Theorem \ref{Main theorem} is inspired by the uniqueness proof in inverse conductivity problems with a single measurement \cite{FI89} and the path argument proposed in \cite{Elschner_Guang-hui_Inverse_elastic_scattering} where the 'non-point-to-point' reflection principle for the Navier equation was applied.
		This paper also provides new uniqueness proofs in determining sound-soft/sound-hard polygonal obstacles with a single incoming wave. In particular, Theorems \ref{Main theorem} and \ref{Th2} remain valid when the surface impedance function is piecewise holomorphic (for instance, piecewise constant) and
		under mixed Dirichlet, Neumann and Robin boundary conditions; see Remark \ref{rem} at the end of the paper.

		
		The remaining part of this paper is organized as follows. In the subsequent two Sections, we recall the
		reflection principle for the Laplacian equation and derive the extension formula for the Helmholtz equation satisfying the Robin boundary condition on a flat subboundary. The proofs of Theorems \ref{Main theorem} and \ref{Th2} will be carried out in Section \ref{sec:4}.

		\section{Reflection principle for Laplacian equation}\label{Harmonic reflection}
		Consider a non-singular real-analytic curve  $\Gamma\subset\R^{2}$ given by $\Gamma:=\{x=(x_{1},x_{2})\in\R^{2}: \ f(x)=0,\,  \nabla f(x)\neq 0\}$ where $f$ is an algebraic function defined on $\R^{2}$ (which means that $f$ is polynomial in $x_1$ and $x_2$ with real coefficients).
		Below we define the mapping $R_{\Gamma}$ with respect to $\Gamma$ by the \emph{Schwarz function} of $\Gamma$ \cite{Schwarz}.  Let $U\subset\R^{2}$ be an open connected set separated into two parts $U^+$ and $U^-$ by $\Gamma$. 
		Now consider a complex domain $V\subset\C^{2}$ such that $V\cap \R^{2}=U$. Using the the bicharacteristic coordinates $z=x_{1}+ix_{2}$, $\omega=x_{1}-ix_{2}$, we obtain the equation of the complexified curve $\Gamma_{\C}$ in the form
		\begin{equation}\label{Complexified curve equation}
			\Gamma_{\C}:=\bigg\{(z, \omega)\in \C^2: f\lb \frac{z+\omega}{2},\frac{z-\omega}{2i}\rb=0\bigg\}.
		\end{equation}
		Note that $\Gamma_{\C}$ coincides with $\Gamma$ on those points such that $z=\overline{\omega}$, where the bar denotes the complex conjugate.
		Since $\nabla f\neq 0$ on $\Gamma$, the above Equation \eqref{Complexified curve equation} is uniquelly solvable for $z$ and $w $ in a neighborhood of $\Gamma$ in $\C^2$.
		We denote the corresponding solutions by $\omega=S(z) \ \mbox{and}\ z=\widetilde{S}(\omega)$. Here  the function $S(z)$ is called the {\it{Schwarz function}} of the curve $\Gamma$ and $z=\widetilde{S}(\omega)$ represent the inverse of $S(z)$. The  mapping $R_\Gamma$ is given by the formula (see \cite{Davis79})
		\ben
		R_\Gamma(x)=R_\Gamma(z):=\{y=(y_1,y_2)\in \R^2: y_1+iy_2=\overline{S(z)}\}.
		\enn
		It is well known that $R_\Gamma: U\rightarrow U$ is a conformal mapping permuting $U^+$ and $U^-$ (that is, if $x\in U^\pm$, then $R_\Gamma(x)\in U^\mp$). For a geometric interpretation of $R_\Gamma$ we refer to \cite{Study, Savina_Sternin_Shatalov_Reflection_Helmholtz}.
		With these notations, we state the reflection principle for harmonic functions with the Robin boundary condition given on an {algebraic} curve, which was verified in \cite{Belinskiy_Savina_Schwartz_Harmonic_Robin}. For notational simplicity we write $\partial_j w(x)=\partial_{x_j} w(x)$ for $x=(x_1, x_2)\in \R^2$.
		\begin{proposition}
			\cite{Belinskiy_Savina_Schwartz_Harmonic_Robin}
			\label{Reflection harmonic general case}
			Let $\gamma\subset \Gamma$ be a subset of $\Gamma$ and let $U=U^+\cup\gamma\cup U^-$ be defined as above with $\gamma:=\Gamma\cap U$. 
			Let  $w(x)$ be a solution to
			\begin{align}\label{Harmonic general case}
				\Delta w(x)=0\ \mbox{in}\ U^{+},\quad
				\alpha(x)\PD_{\nu}w+\beta(x)w=0\ \mbox{on}\ \gamma,
			\end{align}
			where $\alpha,\beta$ are holomorphic functions not vanishing simultaneously and $\PD_{\nu}$ denotes the unit normal to $\gamma$ pointing into $U^{-}$. Then $w$ can be extended from $U^+$ to $U$ by defining $\widetilde{w}= w$ in $U^+$ and
			\be\nonumber
			\widetilde{w}(x)
			&:=&w(R_{\Gamma}\lb x\rb)+\frac{1}{2i}\int\limits_{\gamma}^{R_{\Gamma}\lb x\rb}V(y_{1},y_{2};R_{\Gamma}\lb x\rb)\,\left[ \PD_{2}w(y_{1},y_{2})\D y_{1}-\PD_{1}w(y_{1},y_{2})\D y_{2}\right] \\ \label{eH}
			&&-\frac{1}{2i}\int\limits_{\gamma}^{R_{\Gamma}\lb x\rb}w(y_{1},y_{2})\,\left[ \PD_{2}V(y_{1},y_{2};R_{\Gamma}\lb x\rb)\D y_{1}-\PD_{1}V(y_{1},y_{2};R_{\Gamma}\lb x\rb)\D y_{2}\right],
			\en
			for $x\in U^-$,
			where the integral $\int_{\gamma}^{R_{\Gamma}\lb x\rb}$ is calculated on any path joining an arbitrary point on $\gamma$ with the point $R_{\Gamma}(x)$. The function $V(y;\widetilde{x})$ $(y, \tilde{x}\in \R^2)$ is given by $V=V_1-V_2$, where
			\be\label{Exprssion for V}
			V_1(y;\widetilde{x})=1-2\exp\left[-i\int\limits_{S(z_{0})}^{\omega}\frac{\beta(\widetilde{S}(\tau),\tau)\sqrt{\widetilde{S}'(\tau)}}{\alpha(\widetilde{S}(\tau),\tau)}\D \tau\right],\;
			V_2(y;\widetilde{x})= 1-2\exp \left[i\int\limits_{\widetilde{S}(\omega_{0})}^{z}\frac{\beta(\tau,S(\tau))\sqrt{S'(\tau)}}{\alpha(\tau,S(\tau))}\D \tau\right]
			\en
			with $z=y_{1}+iy_{2}$, $\omega=y_{1}-iy_{2}$, $z_{0}=\widetilde{x}_{1}+i\widetilde{x}_{2}$, $\omega_{0}=\widetilde{x}_{1}-i\widetilde{x}_{2}$. The integrals in (\ref{Exprssion for V}) are complex integrals between two points in the complex plane $\C$.
		\end{proposition}
		
		{We remark that, since $\Gamma$ is an algebraic curve, the Schwarz function $S$ and its inverse $\widetilde{S}$ are complex analytic functions on $\C$ with algebraic singularities only.
			The function $V$ is a multi-valued function over $\C^2$ whose singularities coincide with $S$ and $\widetilde{S}$. It is supposed in Proposition
			\ref{Reflection harmonic general case} that the domain $U$ does not contain such singularities. On the other hand, the constructed function $V$ ensures that the integrand appearing in (\ref{eH})
			vanishes on $\Gamma_\C$ (see \cite{Belinskiy_Savina_Schwartz_Harmonic_Robin}).
			Hence, the integral on the right hand side of (\ref{eH}) is independent on the path of integration and the choice of the point on $\gamma$. Although the mapping $R_\Gamma$ is originally constructed in a small neighborhood of $\Gamma_\C$, by the uniqueness theorem for analytic functions the resulting extension formula is valid in the large domain.
			In the special case that $\Gamma$ is a straight line, one can obviously get a non-local extension formula in the real space $\R^2$.}

		{Proposition \ref{Reflection harmonic general case} will be used later to prove Theorem \ref{Th2} for recovering holomorphic impedance coefficients}. 
		{For the purpose of clarity, from now on we shall restrict our discussions to the case when $\alpha\equiv 1$, $\beta\equiv i\,\lambda$ for some constant $\lambda>0$ and when
			the curve $\Gamma\subset\R^{2}$ coincides with the real axis. With these settings,
			we can get a single-valued extension formula with a
			more explicit integral form than (\ref{eH})}.
		Let $\Omega\subset\R^{2}$ be a connected open set which is symmetric with respect to the real axis $L:=\{(x_{1},x_{2})\in\R^{2}:\ x_{2}=0\}$, i.e., $(x_{1},x_{2})\in\Omega$ if and only if $(x_{1},-x_{2})\in\Omega$. Set $\widetilde{L}:=L\cap \Omega$, and denote by  $\Omega^{+}$ and $\Omega^{-}$ the two symmetric parts of $\Omega$ into which $\Omega\setminus \widetilde{L}$ is divided by $L$. 
		
		\begin{corollary}\label{Harmonic reflection corollary}
			Suppose that $w$ is a solution to
			\begin{equation}\label{Harmonic function reflection}
				\Delta w=0\quad \mbox{in}\quad\Omega^{+},\quad
				\PD_{2}w+i\lambda w=0\ \quad\mbox{on}\quad \widetilde{L}.
			\end{equation}
			Then $w$ can be extended as a harmonic function from $\Omega^+$ to $\Omega$.  If $\chi(t):=\left\{\lb \chi_{1}(t),\chi_{2}(t)\rb:\ t\in[0,1]\right\}\subset\Omega^{+}$ is a  {piecewise-smooth} curve joining  $\chi(1):=x\in\Omega$ to  an arbitrary point  $\chi(0)\in \widetilde{L}$,  then the extended function $\widetilde{w}$ is given by
			\begin{equation}\label{Extension fomula harmonic}
				\begin{aligned}
					\widetilde{w}(x)=\begin{cases}
						w(x)\ &\mbox{if}\ x\in \Omega^{+}\cup \widetilde{L},\\
						(\widetilde{\mathcal{D}}w)(x_{1},-x_{2})\ &\mbox{if}\ x\in\Omega^{-},
					\end{cases}
				\end{aligned}
			\end{equation}
			where the operator $\widetilde{\mathcal{D}}$ is defined by
			\be\nonumber
			(\widetilde{\mathcal{D}}w)(x)\!\!\!&:=&\!\!\!w(x)+2i\lambda e^{-i\lambda x_{2}}\int\limits_{0}^{1}e^{-i\lambda \chi_{2}(t)}\cosh[\lambda \lb \chi_{1}(t)-x_{1}\rb]\, \chi_{2}'(t)w(\chi_{1}(t),\chi_{2}(t))\D t \\ \label{Definition of D0}
			\!\!\!\!&-&\!\!\!2\lambda e^{-i\lambda x_{2}}\int\limits_{0}^{1}e^{-i\lambda \chi_{2}(t)}\sinh[\lambda \lb \chi_{1}(t)-x_{1}\rb]\, \chi_{1}'(t) w(\chi_{1}(t),\chi_{2}(t))\D t\\ \nonumber
			\!\!\!\!&+&\!\!\! 2i e^{-i\lambda x_{2}}\int\limits_{0}^{1}e^{-i\lambda \chi_{2}(t)}\sinh[\lambda\lb \chi_{1}(t)-x_{1}\rb]\,\left [\PD_{2}w(\chi_{1}(t),\chi_{2}(t))\chi_{1}'(t)-\PD_{1}w(\chi_{1}(t),
			\chi_{2}(t))\chi_{2}'(t)\right]\D t.
			\en
			Here $\sinh(t):=(e^{t}-e^{-t})/2$ and $\cosh (t):=(e^{t}-e^{-t})/2$ are hyperbolic functions.
			\begin{proof} Let $x=(x_{1},x_{2})\in\Omega^{+}$, so that $(x_1, -x_2)\in \Omega^-$.
				By Proposition \ref{Reflection harmonic general case},
				\begin{align}\label{Expression for w}
					\begin{aligned} w(x_{1},-x_{2})&=w(x)+\frac{1}{2i}\int\limits_{\widetilde{L}}^{x}
						V(y;x)[ \PD_{2}w(y)\D y_{1}-\PD_{1}w(y)\D y_{2}]\\
						& \ \ \ -\frac{1}{2i}\int\limits_{\widetilde{L}}^{x}w(y)\,[ {\PD_{2}}V(y;x)\D y_{1}-{\PD_{1}}V(y;x)\D y_{2}],
					\end{aligned}
				\end{align}
				where the integral $\int_{\widetilde{L}}^{x}$  is independent on the path joining an arbitrary point on $\widetilde{L}$ with $x$, and the function $V=V_1-V_2$ is given by \eqref{Exprssion for V} with $\alpha=1,\beta=i\lambda$ and the Schwarz functions $S(z)=\widetilde{S}(z)=z$. Obviously,
				\ben V_{1}(y_{1},y_{2};x_{1},x_{2})=1-\exp\Bigg[-i\int\limits_{S(z_{0})}^{y_{1}-iy_{2}}i\lambda \D \tau\Bigg],  \quad\ V_{2}(y_{1},y_{2};x_{1},x_{2})=1-\exp\Bigg[i\int\limits_{\widetilde{S}(w_{0})}^{y_{1}+iy_{2}}i\lambda \D \tau\Bigg].
				\enn
				Since $S(z_{0})=z_{0}=x_{1}+ix_{2}$ and $\widetilde{S}(w_{0})=w_{0}=x_{1}-ix_{2}$,  we get
				\ben
				&V_{1}(y;x)=1-2\exp\Big[\lambda \Big((y_{1}-x_{1})-i(y_{2}+x_{2})\Big)\Big],\\	&V_{2}(y;x)=1-2\exp\Big[-\lambda\Big(\lb y_{1}-x_{1}\rb+i\lb y_{2}+x_{2}\rb\Big) \Big],
				\enn
				implying that
				\begin{align*}
					&V(y_{1},y_{2};x_{1},x_{2})=-4e^{-i\lambda \lb y_{2}+x_{2}\rb} \sinh[\lambda\lb y_{1}-x_{1}\rb], \\ &\PD_{y_1}V(y_{1},y_{2};x_{1},x_{2})=-4\lambda e^{-i\lambda\lb y_{2}+x_{2}\rb}\cosh[\lambda\lb y_{1}-x_{1}\rb], \\
					&  \PD_{y_2}V(y_{1},y_{2};x_{1},x_{2})=4i\lambda e^{-i\lambda\lb y_{2}+x_{2}\rb}\sinh[\lambda\lb y_{1}-x_{1}\rb].
				\end{align*}
				Now let  $\chi(t):=\lb \chi_{1}(t),\chi_{2}(t)\rb$ be a path connecting an arbitrary point $\lb \overline{x}_{1},0\rb\in \widetilde{L}$ with the point $x:=(x_{1},x_{2})\in \Omega^{+}$ such that $\chi(0)=\lb \overline{x}_{1},0\rb\in\widetilde{L}$ and $\chi(1)=x$. With this choice of $\chi(t)$ and the expression for $V$, the right hand side of \eqref{Expression for w} can be rewritten as
				\begin{align*}
					&w(x_{1},-x_{2})=w(x_{1},x_{2}) -\frac{1}{2i}\int\limits_{0}^{1}w(\chi_{1}(t),\chi_{2}(t))\lb 4i\lambda e^{-i\lambda\lb \chi_{2}(t)+x_{2}\rb}\sinh\lambda \lb \chi_{1}(t)-x_{1}\rb\rb\chi_{1}'(t)\D t\\
					&\ \ +\frac{1}{2i}\int\limits_{0}^{1}w(\chi_{1}(t),\chi_{2}(t))\lb -4\lambda e^{-i\lambda \lb \chi_{2}(t)+x_{2}\rb}\cosh\lambda\lb \chi_{1}(t)-x_{1}\rb\rb \chi_{2}'(t)\D t\\
					&\ \  +\frac{1}{2i}\int\limits_{0}^{1}\lb -4e^{-i\lambda \lb \chi_{2}(t)+x_{2}\rb} \sinh\lambda\lb \chi_{1}(t)-x_{1}\rb\rb \lb \PD_{2}w(\chi_{1}(t),\chi_{2}(t)) \chi'_{1}(t)-\PD_{1}w(\chi_{1}(t),\chi_{2}(t))\chi'_{2}(t)\rb\D t.
				\end{align*}
				Simple calculations show that for $x\in\Omega^{+}$,
				\be\nonumber
				&&w(x_{1},-x_{2})=w(x)-2\lambda e^{-i\lambda x_{2}}\int\limits_{0}^{1}e^{-i\lambda\chi_{2}(t)}w(\chi_{1}(t),\chi_{2}(t))  \sinh[\lambda \lb \chi_{1}(t)-x_{1}\rb]\,\chi_{1}'(t)\D t\\ \nonumber
				&&\ \ \ +2i\lambda e^{-i\lambda x_{2}}\int\limits_{0}^{1}e^{-i\lambda  \chi_{2}(t)}w(\chi_{1}(t),\chi_{2}(t)) \cosh[\lambda\lb \chi_{1}(t)-x_{1}\rb]\, \chi_{2}'(t)\D t\\ \nonumber
				&& \ \ \ +2ie^{-i\lambda x_{2}}\int\limits_{0}^{1} e^{-i\lambda  \chi_{2}(t)} \sinh[\lambda\lb \chi_{1}(t)-x_{1}\rb]\,\lb \PD_{2}w(\chi_{1}(t),\chi_{2}(t)) \chi'_{1}(t)-\PD_{1}w(\chi_{1}(t),\chi_{2}(t))\chi'_{2}(t)\rb\D t\\ \label{D0}
				&&  =: (\widetilde{\mathcal{D}}w)(x).
				\en
				Hence the function $\widetilde{w}$ given by (\ref{Extension fomula harmonic})
				is well defined in $\Omega^-$. Below we show that $\widetilde{w}$ is indeed the extension from $\Omega^{+}$ to $\Omega$.
				For this purpose, we first prove that $\Delta \widetilde{w}=0$ in $\Omega^-$. For $x\in \Omega^{-}$, we have
				\begin{align*} \widetilde{w}(x)&=(\widetilde{\mathcal{D}}w)(x_{1},-x_{2}):=w(x_{1},-x_{2})+2i\lambda e^{i\lambda x_{2}}\int\limits_{0}^{1}e^{-i\lambda  \chi_{2}(t)}w(\chi_{1}(t),\chi_{2}(t)) \cosh[\lambda\lb \chi_{1}(t)-x_{1}\rb] \chi_{2}'(t)\D t\\
					&\ \ -2\lambda e^{i\lambda x_{2}}\int\limits_{0}^{1}e^{-i\lambda\chi_{2}(t)}w(\chi_{1}(t),\chi_{2}(t))  \sinh[\lambda \lb \chi_{1}(t)-x_{1}\rb]\chi_{1}'(t)\D t\\
					& \ \ +2ie^{i\lambda x_{2}}\int\limits_{0}^{1} e^{-i\lambda  \chi_{2}(t)} \sinh[\lambda\lb \chi_{1}(t)-x_{1}\rb]\Big( \PD_{2}w(\chi_{1}(t),\chi_{2}(t)) \chi'_{1}(t)-\PD_{1}w(\chi_{1}(t),\chi_{2}(t))\chi'_{2}(t)\Big)\D t.
				\end{align*}
				Taking Laplacian on both sides of the above equation and using $v''(\lambda t)=\lambda^2 v(\lambda t)$ for $v=\sinh,\cosh$, we have
				\begin{align*}
					\Delta \widetilde{w}(x_{1},x_{2})=\Delta w(x_{1},-x_{2})=0\quad\mbox{for all}\quad x\in \Omega^-.
				\end{align*}
				Next we show that $\widetilde{w}$ and $\partial_2\widetilde{w}$ are continuous on $\tilde{L}$. For $x=(x_1,x_2)\in \Omega^-$ with $x_2$ close to zero,
				we can choose the finite line segment $\chi(t)=(\chi_{1}(t),\chi_{2}(t)):=(x_{1},-tx_{2})\subset \Omega^{+}$ for $t\in [0,1]$, which is perpendicular to the interface $\widetilde{L}$ at $(x_1,0)$.  This is due to the reason that the integral does not depend on the path of integration.
				Using this curve in \eqref{Extension fomula harmonic}, we get
				\begin{align*}
					\widetilde{w}(x)=
					w(x_{1},-x_{2})-2i\lambda e^{-i\lambda x_{2}}\int\limits_{0}^{1}e^{i\lambda tx_{2}}w(x_{1},-tx_{2})x_{2}\D t, \end{align*}
				for $x\in\Omega^{-}$ with $x_{2}$ sufficiently small.
				Changing the variable $-tx_{2}=s$ in the above expression, we get (cf. \cite[Remark 4.1]{Belinskiy_Savina_Schwartz_Harmonic_Robin})
				\be\label{w0}
				\widetilde{w}(x_1,x_2)=
				w(x_{1},-x_{2})+2i\lambda e^{-i\lambda x_{2}}\int\limits_{0}^{-x_{2}}e^{-i\lambda s}w(x_{1},s)\D s,
				\en
				which implies the continuity of $\widetilde{w}$ on $\widetilde{L}$. Taking the derivative with respect to $x_2$, we obtain
				\[\PD_{2}\widetilde{w}(x_{1},x_{2})=-\PD_{2}w(x_{1},-x_{2})-2\lambda^{2}e^{i\lambda x_{2}} \int\limits_{0}^{-x_{2}}e^{-i\lambda s}w(x_{1},s)\D s-2i\lambda e^{2i\lambda x_{2}}w(x_{1},-x_{2}), \]
				for $x\in\Omega^{-}$ close to $\widetilde{L}$, and thus
				\begin{align*}
					\lim_{x_{2}\rightarrow 0^+}\PD_{2}\widetilde{w}(x_{1},x_{2})=
					-\PD_{2}w(x_{1},0)-2i\lambda w(x_{1},0),\qquad \mbox{if} \ (x_{1},x_{2})\in\Omega^{-}.
				\end{align*}
				Using Equation \eqref{Harmonic function reflection}, we conclude that
				$\PD_{2} \widetilde{w}\in C(\Omega)$.
				Hence, the Cauchy data of $\widetilde{w}$ are continuous on $\widetilde{L}$, implying that $\widetilde{w}$ is harmonic over $\Omega$.
				In particular,
				\ben
				\lim_{x_{2}\rightarrow 0^-}(\PD_{2}\widetilde{w}+i\lambda \widetilde{w})=-\PD_{2}w(x_{1},0)-i\lambda w(x_{1},0)=0.
				\enn
				This proves that
				$\widetilde{w}$ given by \eqref{Extension fomula harmonic} is indeed the extension formula for $w$.
			\end{proof}
		\end{corollary}
		\begin{remark}\label{remark2.3}
			\begin{itemize}
				\item[(i)]
				Note that the operator $\widetilde{\mathcal{D}}$ defined in (\ref{D0})
				maps analytic functions in $\Omega^{+}$ to itself and it can be extended to an operator defined on a class of functions over $\Omega$. From the proof of Corollary \ref{Harmonic reflection corollary}, it holds that $\widetilde{w}(R_L x)=\widetilde{\mathcal{D}}\widetilde{w}(x)$ for all $x\in\Omega$, where $R_L$ denotes the reflection with respect to $L$.
				\item[(ii)] Suppose that $\Omega^+$ fulfills the following geometric condition: the projection onto $L$ of any line segment lying in $\Omega^+$ is a subset of $\widetilde{L}$, that is, $(x_1, x_2)\in \Omega^+$ implies $(x_1,0)\in \widetilde{L}$. Then the extension formula (\ref{Definition of D0}) can be reduced to (\ref{w0}) by choosing the path $\chi_1(t)\equiv x_1$ and $\chi_2(t)=tx_2$ for $t\in[0,1]$.
			\end{itemize}
		\end{remark}
		\section{reflection principle for Helmholtz equation}\label{Helmholtz reflection}
		
		There have been extensive works in the literature related to reflection principles for the  elliptic partial differential equations of second order with constant or real analytic coefficients.  We refer to \cite{Belinskiy_Savina_Schwartz_Harmonic_Robin,
			Davis79,Ebenfelt_Khavinson_Point to point refelection_Harmonic_Dirichlet,
			Garabedian,Khavinson_Shapiro_Harmonic_Remark} for references related to harmonic functions.
		In \cite{Savina_Sternin_Shatalov_Reflection_Helmholtz, Savina_general_elliptic_Transaction_AMS} the reflection principle  for the Helmholtz equation with the Dirichlet boundary condition given on a real analytic curve was derived, which was later extended to the Neumann boundary condition in \cite{Savina_Helmholtz_Neumann}.
		In two dimensions, we believe that the reflection principle for the Helmholtz equation subject to the Robin boundary condition can be established analogously, following the approaches for harmonic functions \cite{Belinskiy_Savina_Schwartz_Harmonic_Robin} and the ideas in \cite{Savina_Sternin_Shatalov_Reflection_Helmholtz,
			Savina_general_elliptic_Transaction_AMS,
			Savina_Helmholtz_Neumann}. {In fact, for this purpose it is only required to redefine the function $V$ appearing in (\ref{eH}) subject to the Helmholtz equation}. However, in this paper we prefer an alternative method presented in our previous paper \cite{Elschner_Guang-hui_Inverse_elastic_scattering}, where a non-local extension formula for the Navier equation with the Dirichlet boundary condition was established in 2D.
		By \cite{Elschner_Guang-hui_Inverse_elastic_scattering}, the existence of an extension operator for the Helmholtz equation relies on the case of the harmonic function.
		
		Let $\widetilde{L}\subset L$ and $\Omega=\Omega^+\cup\widetilde{L}\cup\Omega^-$ be defined in the same way as before. We remark that the domain $\Omega^+$ can be any connected open set lying on one side of $L$ and having the flat subboundary $\widetilde{L}$ on its boundary.
		Consider the Helmholtz equation subject to the impedance boundary condition
		\begin{equation}\label{Helmholtz equation for Reflection}
			\Delta u+k^{2}u=0\ \quad\mbox{in}\;\Omega^{+},\quad
			\PD_{2}u+i\lambda u=0\quad\mbox{on}\; \widetilde{L},
		\end{equation}
		where $k^2>0$ and $\lambda>0$ are constants.  Our aim is to extend the solution $u$ of Equation \eqref{Helmholtz equation for Reflection}  from $\Omega^{+}$ to $\Omega^{-}$. 
		The following theorem is sufficient in our uniqueness proofs.
		\begin{theorem}\label{Reflection principle for Helmholtz equation}
			The solution $u$ to Equation \eqref{Helmholtz equation for Reflection} can be extended from $\Omega^{+}$  to $\Omega^{-}$ by some extension operator $\mathcal{D}$ (see \eqref{Definition of Dk} below), that is, defining
			\begin{align*}
				u(x_{1},-x_{2}):=(\mathcal{D}u)(x_1, x_2)\qquad\mbox{for} \ (x_{1},x_{2})\in\Omega^{+},
			\end{align*}
			we have $\Delta u+k^2 u=0$ in $\Omega$.
		\end{theorem}
		\begin{proof}
			Let $\Phi(x;y)=-\frac{1}{2\pi}\log\lvert x-y\rvert$, with $x,y\in \R^2$ and $x\neq y$, be the fundamental solution to $-\Delta$ in $\R^2$. Denote by $G_{0}(x;y)$ the Green's function to $-\Delta$ in the half plane $\R_+^2:= \{x: x_2>0\}$ subject to the Robin boundary condition (\ref{Helmholtz equation for Reflection}) on $x_2=0$. By \cite[Corollary 2.2]{Elschner_Guang-hui_Inverse_elastic_scattering},
			$G_0(x;y)$ can be constructed from  $\Phi(x;y)$ by
			\ben
			G_{0}(x;y)=\Phi(x;y)+\widetilde{\mathcal{D}}_x\Phi(R_{L}x;y),\quad x,y\in \R_+^2,\; x\neq y,
			\enn
			where $\widetilde{\mathcal{D}}$ is the extension operator for $-\Delta$ defined in (\ref{Definition of D0}). Here we write $\widetilde{\mathcal{D}}=\widetilde{\mathcal{D}}_{x}$ to indicate the action of the integral operator with respect to the variable $x$.
			For a solution $u$ to \eqref{Helmholtz equation for Reflection}, we introduce a new function
			\begin{align}\label{Definition of w}
				w(x):= u(x)-k^{2}\int\limits_{\Omega^{+}}G_{0}(x;y)u(y)\D y,\qquad \quad x \in\Omega^{+}.
			\end{align}
			It then follows that $w$ satisfies the following boundary value problem of the harmonic equation
			\begin{equation}\label{Equation for w}
				\Delta w=0\quad\mbox{in}\quad\Omega^{+},\qquad
				\PD_{2}w+i\lambda w=0\quad\mbox{on}\quad \widetilde{L}.
			\end{equation}
			Applying Corollary \ref{Harmonic reflection corollary},  $w$ can be extended as a harmonic function defined in all of $\Omega$. Denoting by $w$ the extended function, we have
			\ben
			w(x_1, x_2)=(\widetilde{\mathcal{D}}w)(x_{1},-x_{2}),\quad
			w(x_1, -x_2)=(\widetilde{\mathcal{D}}w)(x_{1},x_{2})
			\quad\mbox{for all}\quad x\in \Omega.
			\enn
			Inserting \eqref{Definition of w} into the second identity of the previous relations, we get
			\begin{align}\nonumber
				u(x_{1},-x_{2})&=(\widetilde{\mathcal{D}}u)(x)-k^{2}\int\limits_{\Omega^{+}}
				\widetilde{\mathcal{D}}_{x}G_{0}(x;y)\,u(y)\D y +k^{2}\int\limits_{\Omega^{+}}G_{0}(x_{1},-x_{2};y)u(y)\D y\\ \label{Definition of Dk}
				&:=(\mathcal{D}u)(x)
			\end{align}
			for all $x\in\Omega^{+}$. Since the right hand side of (\ref{Definition of Dk}) depends only on $u|_{\Omega^+}$, the operator $\mathcal{D}$ can be regarded as an extension operator for the Helmholtz equation. This finishes the proof of Theorem \ref{Reflection principle for Helmholtz equation}.
		\end{proof}
		{In the special case that $\lambda\equiv 0$ (Neumann boundary condition), we have $\widetilde{D}=\widetilde{D}_x\equiv I$ and $G_0(x;y)=\Phi(x_1,x_2;y)+\Phi(x_1,-x_2;y)$ for $x=(x_1, x_2)$, $y=(y_1, y_2)\in \R^2_+$. This implies that $\widetilde{D}_x G_0(x;y)=G_0(x_1,-x_2;y)$. Hence, we obtain from (\ref{Definition of Dk}) that $\mathcal{D}u=\widetilde{\mathcal{D}}u=u$.
		}

		\begin{remark} The integral domain $\Omega^+$ in (\ref{Definition of Dk}) can be replaced by any connected domain $K\subset \Omega^+$ containing $x$ such that $\partial K\cap\widetilde{L}\neq\emptyset$. In fact, for such $K$ we can always choose a curve $\chi(t)\subset K$ connecting a point on $\partial K\cap\widetilde{L}$ with $x$, so that $\widetilde{\mathcal{D}}$ can be well defined with the path $\chi(t)$.
			Since $\widetilde{\mathcal{D}}_{x}G_{0}(x;y)=G_{0}(x_1,-x_2;y)$ for all $y\in \Omega^+\backslash\overline{K}$, we have
			\ben
			&&-k^{2}\int\limits_{\Omega^{+}}
			\widetilde{\mathcal{D}}_{x}G_{0}(x;y)\,u(y)\D y +k^{2}\int\limits_{\Omega^{+}}G_{0}(x_{1},-x_{2};y)u(y)\D y \\
			&=& -k^{2}\int\limits_{K}
			\widetilde{\mathcal{D}}_{x}G_{0}(x;y)\,u(y)\D y +k^{2}\int\limits_{K}G_{0}(x_{1},-x_{2};y)u(y)\D y,
			\enn from which the assertion follows.
			{Note that both $\widetilde{\mathcal{D}}_{x}G_{0}(x;y)$ and $G_{0}(x_1,-x_2;y)$ are singular for $y\in \Omega^+$ getting close to the curve $\chi(t)\subset \Omega^+$}.
		\end{remark}
		\begin{remark}
			Theorem \ref{Reflection principle for Helmholtz equation} is valid for all $k, \lambda\in \C$ in two dimensions. However, in 3D the analogue of Theorem \ref{Reflection principle for Helmholtz equation}
			does not hold if $\Omega^+\subset \R^3$ is nonconvex and $k=ia, \lambda=-ic$ with $0<c<a$; see the example constructed in \cite{Diaz_Ludford-Reflection principle}. Under the geometric condition of $\Omega^+$ stated in Remark \ref{remark2.3} (ii), the extension operator takes the same form as the Laplacian equation for all $k, \lambda\in \C$ in $\R^n$ ($n\geq 2$) (see \cite{Diaz_Ludford-Reflection principle}).
		\end{remark}

		As an application of the extension operator, we prove that solutions of the Helmholtz equation in a sector with the Robin boundary condition can be extended to the whole space. To state this result, we denote by $(r,\theta)$ the polar coordinates in $\R^2$.
		Given $\theta_0\in(0, \pi/2]$, set
		\be\label{halfline}
		\Sigma_j:=\{(r,\theta): r>0,\theta\in (j\theta_0, (j+1)\theta_0) \},\quad  L_j:=\{(r,\theta): r>0,\theta= j\theta_0\}
		\en
		for $j=0,1,2,\cdots$. Obviously, we have $\partial \Sigma_j=L_j\cup L_{j+1}$. Suppose that the normal direction $\nu$ at $L_j$ is directed into $\Sigma_j$.
		
		\begin{lemma}\label{Lemma about non-existence of non-parallel lines}
			Let $K\supset \Sigma_0$ be an unbounded domain. Suppose that
			$u$ is a solution to the Helmholtz equation $(\Delta+k^2)u=0$ in $K$ satisfying the impedance boundary condition $\partial_\nu u+i\eta_j u=0$ with the constant $\eta_j\in \C$ on $L_j$ for $j=0,1$. Then $u$ can be analytically extended onto the whole space $\R^2$.
		\end{lemma}
		
		\begin{proof} By Theorem \ref{Reflection principle for Helmholtz equation}, $u$ can be analytically extended from $\Sigma_0$ to $\Sigma_1$ by
			\ben
			u(x)=(\mathcal{D}_{1}u)(R_{L_1}x)\quad\mbox{for all}\quad x\in \Sigma_1,
			\enn
			where $\mathcal{D}_{1}$ is the extension operator for the Helmholtz equation corresponding to  boundary condition  $\partial_\nu u+ i\eta_1 u=0$ on $L_1$. Moreover, it holds that $u(R_{L_1}x)=(\mathcal{D}_{1}u)(x)$ for all $x\in \Sigma_0\cup L_1\cup \Sigma_1$. The function $v(x):=u(R_{L_1}x)$ for $x\in\Sigma_1$ satisfies the boundary value problem
			\ben
			\Delta v+k^2v=0\quad\mbox{in}\quad \Sigma_1,\quad \partial_\nu v+i\eta_1 v=0\quad\mbox{on}\quad L_1,\quad \partial_\nu v-i\eta_2 v=0\quad\mbox{on}\quad L_2.
			\enn
			Applying Theorem \ref{Reflection principle for Helmholtz equation} again, we can extend $v$ from $\Sigma_1$ to $\Sigma_2$ by
			\ben
			v(x)=\mathcal{D}_{2} v(R_{L_2}x)=\mathcal{D}_{2}u(R_{L_1}R_{L_2}x)
			\quad\mbox{for all}\quad x\in \Sigma_2,
			\enn
			where $\mathcal{D}_{2}$ is the extension operator corresponding to the boundary operator $\partial_\nu -i\eta_2$.
			Since $u=\mathcal{D}_{1}v$ in $\Sigma_1$, one can extend $u$ from $\Sigma_1$ to $\Sigma_2$ by
			\ben
			u(x)=\mathcal{D}_{1}v(x)=\mathcal{D}_{1}\mathcal{D}_{2}v(R_{L_2}x)
			=\mathcal{D}_{1}\mathcal{D}_{2}u(R_{L_1}R_{L_2}x)
			\enn
			for all $x\in \Sigma_2$. {Repeating this process, we may extend $u$ from $\Sigma_0$ to the upper half plane $x_2\geq0$. In particular, this implies that
				$\partial_\nu u+i\eta_0 u=0$ on $x_2=0$.} Applying Theorem \ref{Reflection principle for Helmholtz equation} again we can extend $u$ from $\R^2_+$ to the whole space $\R^2$.
		\end{proof}

		\begin{remark}\label{rem:3.5}  
			The results in Theorem \ref{Reflection principle for Helmholtz equation} and Lemma \ref{Lemma about non-existence of non-parallel lines} are well known in the Dirichlet and Neumann cases. The Dirichlet boundary condition corresponds to the odd reflection $u(R_Lx)=-u(x)$ (i.e., $\mathcal{D}u=-u$), whereas the Neumann boundary condition gives the even reflection $u(R_Lx)=u(x)$  ( i.e., $\mathcal{D}u=u$).
			Hence, Lemma \ref{Lemma about non-existence of non-parallel lines} remains valid under mixed Dirichlet, Neumann and Robin boundary conditions
			imposed on $L_0\cup L_1$.
		\end{remark}
		The following results will be repeatedly used in our uniqueness proof, which follow from
		Theorem \ref{Reflection principle for Helmholtz equation}
		and Lemma \ref{Lemma about non-existence of non-parallel lines} by slight modifications.
		
		\begin{lemma}\label{lem:3.6} Let $\Omega\subset \R^2$ be a bounded connected polygonal domain and suppose that $\lambda(x)$ is a piecewise constant function on $\partial \Omega$, satisfying $\lambda\equiv {\rm Const}$ on the line segment $l\subset \partial \Omega$.
			Denote by $L$ the straight line containing $l$, $R_L(\Omega)$ the symmetric domain of $\Omega$ with respect to $L$ and by $\mathcal{D}$ the constructed extension operator corresponding to the boundary condition on $l$ in
			Theorem \ref{Reflection principle for Helmholtz equation}.
			Consider the Helmholtz equation
			\be\label{eq:robin}
			\Delta u+k^2 u=0\quad\mbox{in}\quad \Omega\cup l\cup U, \qquad \partial_\nu u+i\lambda(x) u=0\quad\mbox{on}\quad \partial \Omega,
			\en
			where $U\cap \Omega=\emptyset$ is an unbounded connected domain with $l\subset \partial U$. We have
			
			(i) The function $v:=\mathcal{D}u$ is a well-defined analytic function in $U\cup R_L(\Omega)$ and satisfies
			\ben
			\Delta v+k^2 v=0\quad\mbox{in}\quad U\cup R_L(\Omega), \quad
			\partial_\nu v+i\lambda(R_Lx)\, v=0 \quad\mbox{on}\quad \partial R_L(\Omega).
			\enn
			Here the normal direction on $\partial R_L(\Omega)$ is obtained by reflecting the counterpart on $\partial \Omega$ with respect to the straight line parallel to $L$ and passing through the origin $(0,0)\in \R^2$.
			
			(ii) The functions $v$ and $u$ can be analytically extended to the whole space under one of the following conditions:
			\begin{itemize}
				\item[(a)] There exists one side  of $\partial R_L(\Omega)$ extending to a whole straight line in $U$. Moreover, one half plane divided by this straight line is a subset of $U$.
				\item[(b)] There exist two neighboring sides of $\partial R_L(\Omega)$, each of them extending to a half-line in $U$. Moreover, the infinite sector formed by these two half lines is a subset of $U$.
			\end{itemize}
			
		\end{lemma}
		
		\begin{proof}
			(i) By Theorem \ref{Reflection principle for Helmholtz equation}, $v$ satisfies the Helmholz equation in $U$ and $u(R_Lx)=\mathcal{D}u(x)$ at least for $x\in U\cup \Omega$ close to the interface $l$. Hence, $v(x)=u(R_L x)$ for $x\in U$ near $l$. Since the function $x\rightarrow u(R_L x)$ can be extended to $R_L(\Omega)$, $v$ can also be extended from $U$ to $R_L(\Omega)$, if $R_L(\Omega)$ is not a subset of $U$. In the case $R_L(\Omega)\subset U$, the assertion follows directly from Theorem \ref{Reflection principle for Helmholtz equation}.
			
			(ii) In case (a), $v$ is a solution to the Helmholtz equation in an unbounded domain containing the half plane and satisfies the Robin boundary condition on the boundary. Applying Theorem \ref{Reflection principle for Helmholtz equation}, $v$ can be analytically extended onto $\R^2$. The global extension also applies to $u$, since $u(x)=v(R_L x)$ for $x$ near $l$. The case (b) can be proved analogously using Lemma \ref{Lemma about non-existence of non-parallel lines}.
		\end{proof}
		
		In our applications of Lemma \ref{lem:3.6} (see Case (ii) in the proof of Theorem \ref{Main theorem} below), the polygon $\Omega$ is taken as the gap domain between two different obstacles, $u$ the total field corresponding to one of these two obstacles, $\Omega\cup l\cup U$ the complement of this obstacle in $\R^2$ and $l$ a Robin level set of $u$. We refer to Figure \ref{f1} for two simple examples corresponding to cases (a) and (b) in Lemma \ref{lem:3.6} (ii).
		In (\ref{eq:robin}) we neglect the normal directions, since $\lambda$ can be taken as an arbitrary constant on each side of $\partial \Omega$.

		\section{Proof of main results}\label{sec:4}

		To prove Theorem \ref{Main theorem},
		we shall combine the arguments in \cite{Elshner_Yamamoto_polygonal} for treating the Neumann boundary condition and those in elasticity \cite{Elschner_Guang-hui_Inverse_elastic_scattering} where a \textquoteleft non-point-to-point\textquoteright \ reflection principle for the Navier equation was applied to determine a connected rigid polygon. {In contrast with the Dirichlet case \cite{Rondi05,CY,LHZ06, Elschner_Guang-hui_Inverse_elastic_scattering}, the Robin level set (a curve on which the Robin boundary condition is satisfied) of the total field can be unbounded}. We will adapt the arguments of \cite{Elshner_Yamamoto_polygonal} by investigating two Robin half-lines starting from a corner point in our uniqueness proof.

		\subsection{Proof of Theorem \ref{Main theorem}}
		Theorem \ref{Main theorem} will be proved by contradiction argument. We only need to consider the determination of the shape, because the recovery of the Robin coefficients follows directly from Holmgren's uniqueness theorem.
		Suppose that $D_{1}\neq D_{2}$, but the corresponding far-field patterns are identical, i.e.,
		$u_{1}^{\infty}(\widehat{x},d)=u_{2}^{\infty}(\widehat{x},d)$  for all $\widehat{x}\in \Sb$ and for fixed $d \in \Sb$.
		Using  the Rellich's lemma, we have that
		\begin{equation}\label{Equality of total field}
			u_{1}=u_{2}\   \mbox{in the unbounded component $E$ of $\R^{2}\setminus{\overline{D_{1}\cup D_{2}}}$}.
		\end{equation}
		Since both $D_1$ and $D_2$ are polygons, using (\ref{Equality of total field}) {and interchanging $u_1$ and $u_2$ if necessary},
		we only need to consider the following cases
		(see \cite[Lemma 7]{Elshner_Yamamoto_polygonal} for the details):
		
		\begin{description}
			\item[Case (i)]
			There are two half-lines $L_0$ and $L_1$ starting from a corner point $O\in  \partial D_1\cap D_2^c\cap \partial E$ such that $\partial_\nu u_2+i \lambda_2 u_2=0$ on $L_0\cup L_1$; see Figure \ref{f2}.
			Here $D_j^c:=\R^2\backslash\overline{D}_j$ for $j=1,2$ and we may neglect the normal directions at $L_0$ and $L_1$.
		\end{description}
		
		\begin{figure}[htbp]
			\centering
			\includegraphics[width=5.5cm,height=4cm]{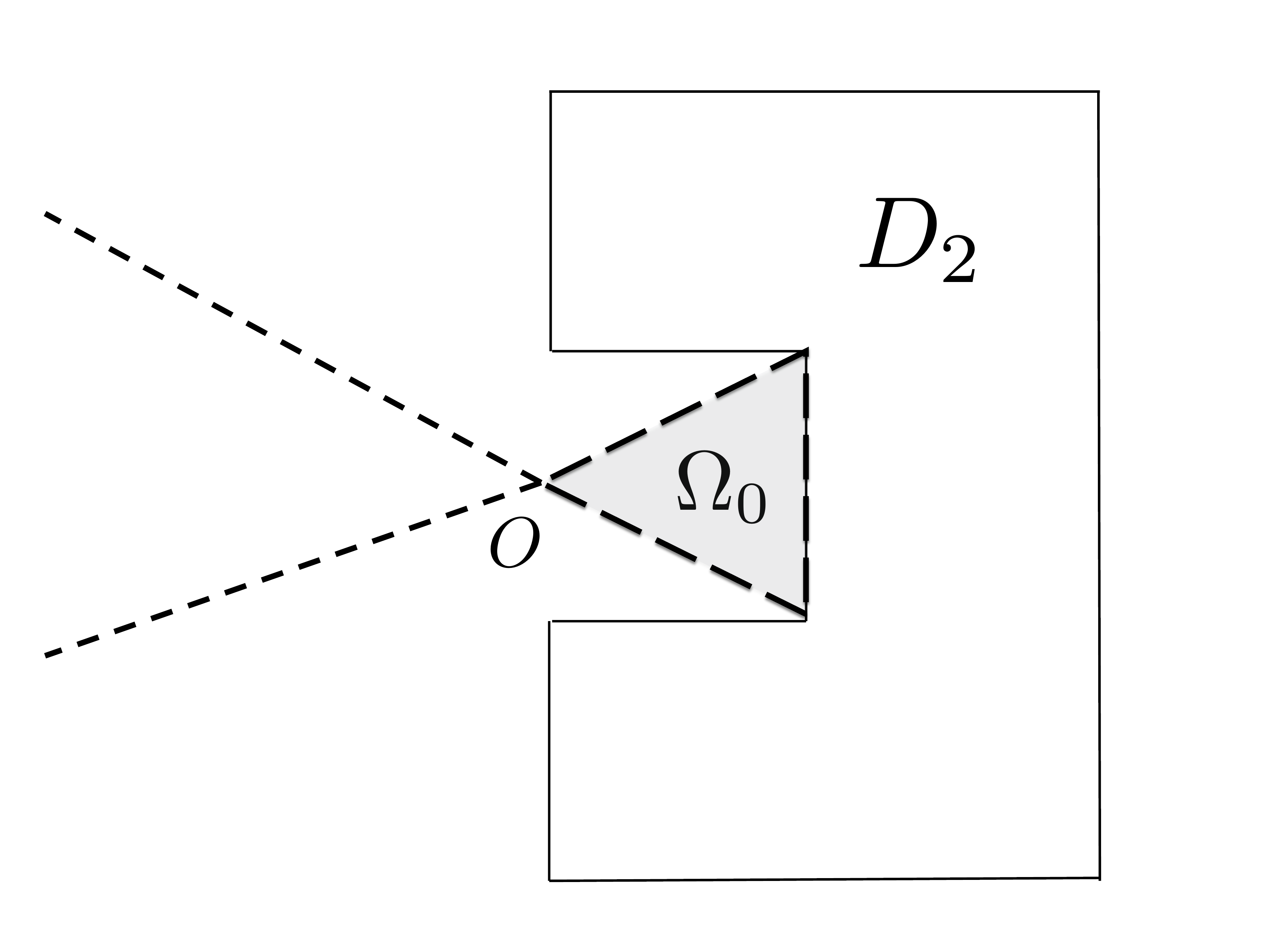}
			\caption{Illustration of two different polygonal obstacles $D_2$ and $D_1$, where $\Omega_0$ denotes the gap domain between $D_1$ and $D_2$. There are two sides of $D_1$ around the corner $O$, each of them can be extended to a half-line in $D_2^c$. }
			\label{f2}
		\end{figure}

		Since the Helmholtz equation is rotational invariant, we may suppose that the corner $O$ is located at the origin and that the half-lines $L_0$ and $L_1$ coincide with those defined in (\ref{halfline}). By the connectness of $D_2$,
		the infinite sector $\Sigma_0$ formed by $L_0\cup L_1$ must lie completely in $E\cap D_2^c$.
		This implies that the total field $u_2=\exp(ikx\cdot d)+u^{sc}_2$ is a solution to the Helmholtz equation in a neighboring area of the sector $\Sigma_0$. Applying Lemma \ref{Lemma about non-existence of non-parallel lines}, we can extend $u_2$ and thus $u_2^{sc}$ to the whole space, which gives  $u_2^{sc}\equiv 0$. Consequently, $u_2(x)=e^{ikx\cdot d}$ fulfills the Robin boundary condition on $\partial D_2$. It then follows that
		\be\label{eq:plane}
		\nu(x)\cdot d+\lambda/k =0 \quad\mbox{for all}\quad x\in \partial D_2
		\en
		with fixed $d\in\Sb$ and $k,\lambda>0$. However,
		this is impossible, since the boundary $\partial D_2$ has at least three unit normal directions $\nu^{(j)}\in\Sb$ ($j=1,2,3$) pointing into $D_2^c$ such that $\nu^{(2)}-\nu^{(1)}$ and  $\nu^{(3)}-\nu^{(1)}$ are linearly independent.
		
		\begin{description}
			\item[Case (ii)]
			There exists a finite line segment $l_0\subset E\cap D_2^c$ with both the end points lying on $\partial D_2\cap \partial E$ such that $\partial_\nu u_2+i \lambda_2 u_2=0$ on $l_0$; see Figure \ref{f1}. 
		\end{description}
		
		\begin{figure}[htbp]
			\centering 
			\includegraphics[width=5.5cm,height=4cm]{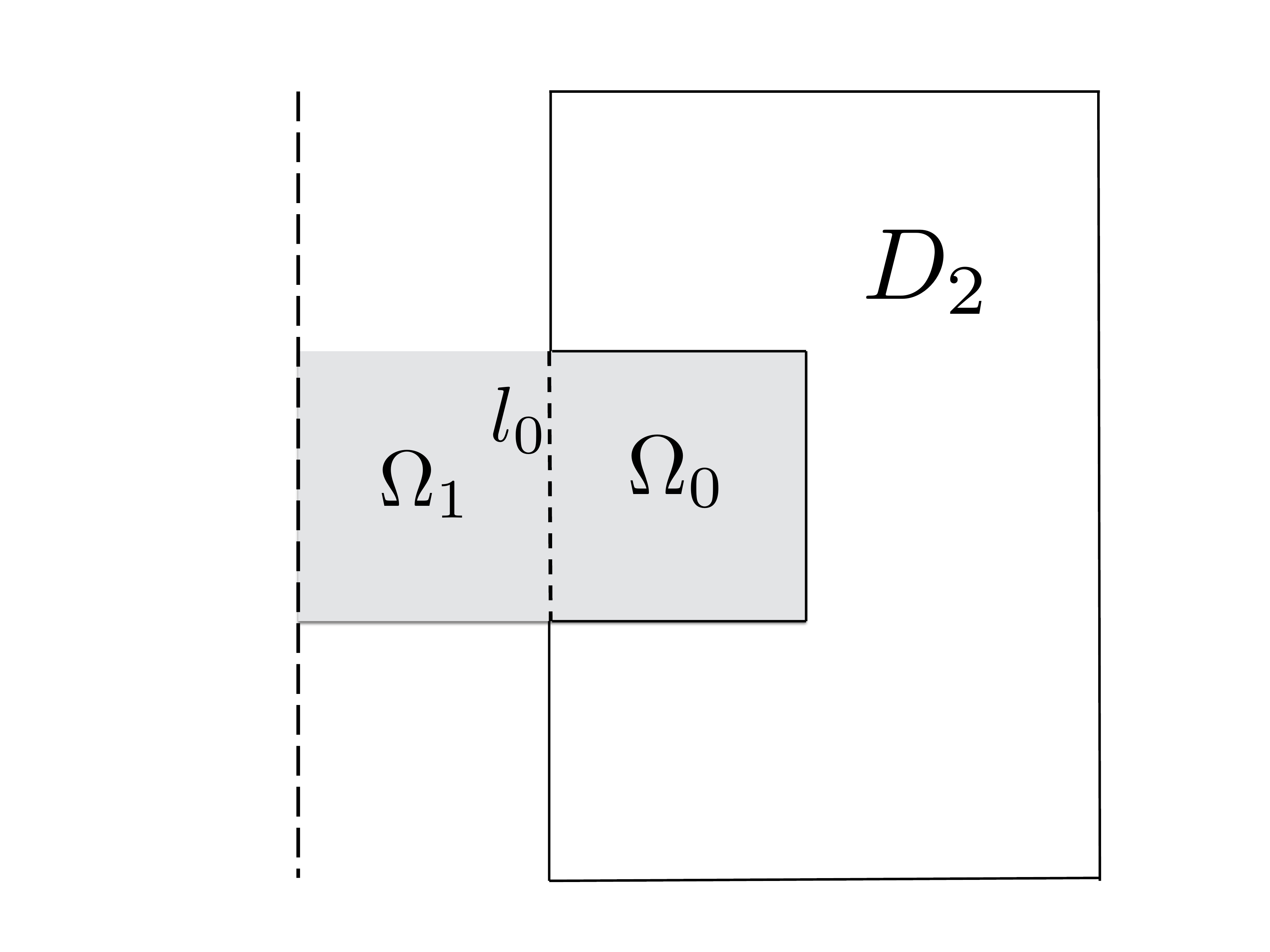}\quad
			\includegraphics[width=5.5cm,height=4cm]{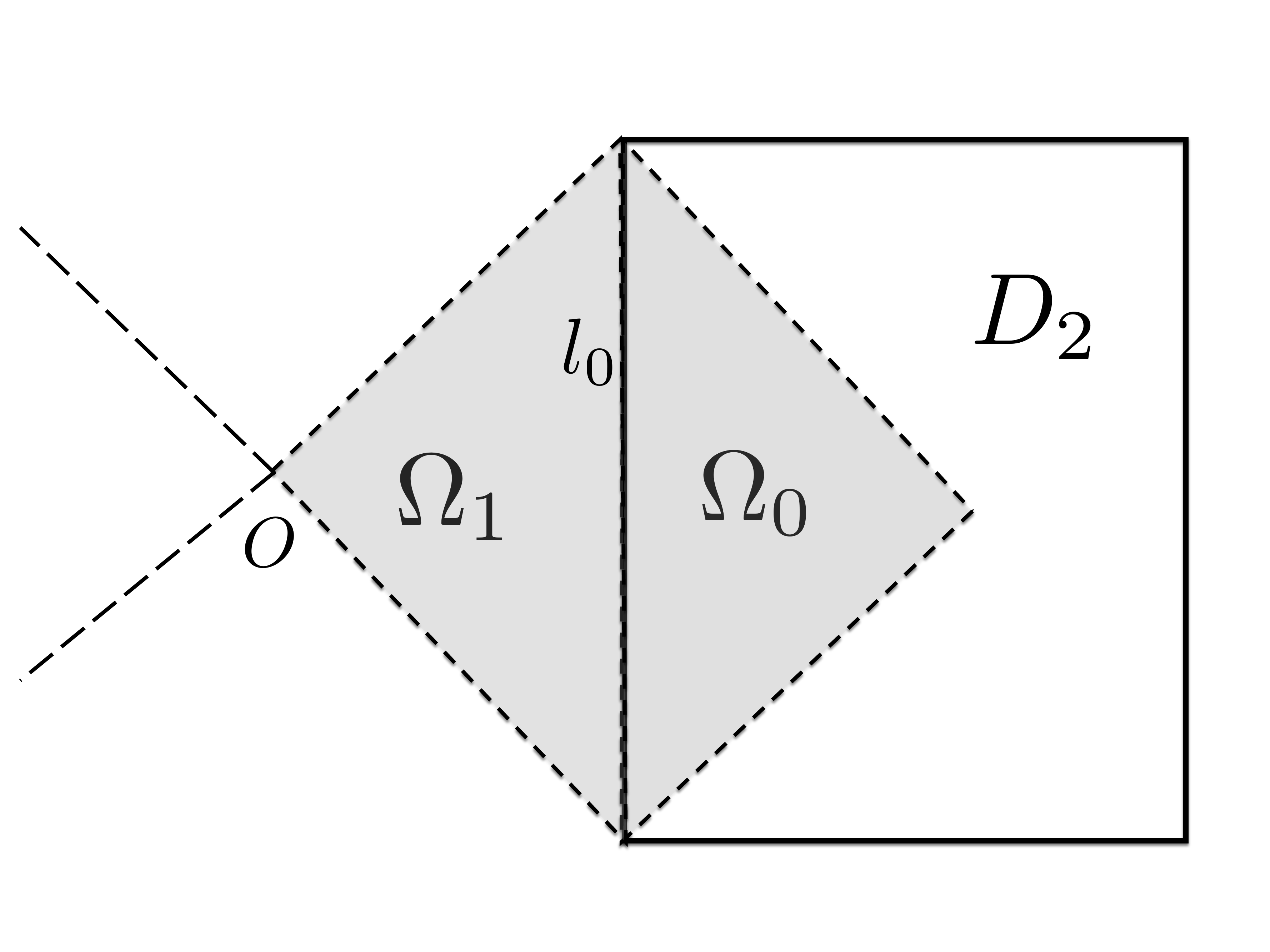}
			\caption{Illustration of two different polygonal obstacles $D_2$ and $D_1$, where $\Omega_0$ denotes the gap domain between $D_1$ and $D_2$.
				$\Omega_1$ denotes the reflection of $\Omega_0$ with respect to $l_0$. Left: One side of $\partial \Omega_1$ extends to a whole straight line in $D_2^c$. Right: Two neighboring sides of $\partial\Omega_2$ extend to half-lines in $D_2^c$ starting from the corner $O$, forming an infinite sector in $D_2^c$.
			}\label{f1}
		\end{figure}

		In Case (ii), we shall apply the approach of \cite{Elschner_Guang-hui_Inverse_elastic_scattering} to prove that the total field $u_2$ can be analytically extended onto $\R^2$, leading to the same contradiction as in Case (i).
		Choose a point  $P_{0} \in l_{0}$ and  a continuous injective path $\gamma(t)$, $t\geq 0$ which starts from $\gamma(0)=P_{0}$ and connects to infinity  in the unbounded component of $D_{2}^{c}\setminus{l_{0}}$. Denote by $\Omega_{0}\subset D_{2}^{c}$ the bounded connected component of $D_{2}^{c}\setminus{l_{0}}$. Since $\PD \Omega_{0}\subset \PD D_{2}\cup l_{0}$, $u_{2}$ is solution to
		\begin{align*}
			\Delta u_{2}+k^{2}u_{2}=0\ \mbox{in}\ \Omega_{0},\qquad \ \mathcal{B}_{0}u_{2}=0\ \mbox{on}\ \PD \Omega_{0},
		\end{align*}
		where $\mathcal{B}_{0}:=\partial_\nu+i\lambda_0(x)$ is a Robin boundary operator with some piecewise constant function $\lambda_0$ defined on $\partial \Omega_0$.
		Note that $\lambda_0\equiv {\rm Const}$  on any side of $\partial \Omega_0$ not containing $l_0$ and may be piecewise constant on the side of $\partial \Omega_0$ containing $l_0$.
		
		We proceed with several notations to be used below. Denote by $R_{n}$  $(n\geq 0)$ the reflection with respect to  the straight line $L_{n}$ containing the line segment $l_{n}$. 
		Let $\mathcal{D}_{n}$ denote the extension operator
		for the Helmholtz equation with  respect to the line segment $l_{n}\subset L_{n}$ and subject to the Robin boundary operator $\mathcal{B}_{n}$ on $l_{n}$.  This can be obtained from Equation \eqref{Definition of Dk} after translation and rotation.
		
		Applying Lemma \ref{lem:3.6}, the function $u^{(0)}_{2}:=u_{2}$ can be extended to the symmetric domain $\Omega_{1}:=R_{0}\lb \Omega_{0}\rb$ of $\Omega_0$ with respect to $L_0$ by the extension operator $\mathcal{D}_{0}$. Setting
		$u_{2}^{(1)}(x):=\mathcal{D}_{0}u_{2}^{(0)}(x)$
		in $\Omega_1\cup D_2^c$, we find that
		$u_{2}^{(1)}(x)=u_{2}^{(0)}(R_0x)$ in $\Omega_1$ and it
		satisfies the boundary value problem
		\begin{align*}
			\Delta u_{2}^{(1)}+k^{2}u_{2}^{(1)}=0\ \mbox{in}\ \Omega_{1}, \qquad \mathcal{B}_{1}u_{2}^{(1)}=0\; \mbox{on}\ \PD \Omega_{1},
		\end{align*}
		for some Robin boundary operator $\mathcal{B}_{1}$.
		By Lemma \ref{lem:3.6}, the Robin coefficient is piecewise constant on $\partial \Omega_1$.
		Since $\Omega_{1}$ is bounded,  we have $\PD \Omega_{1}\cap \{\gamma(t):\ t>0\}\neq \phi$. Set $t_{1}:=\sup \{t:\ \Omega_{1}\cap \gamma(t)\neq \phi\}$. By \cite{Elschner_Guang-hui_Inverse_elastic_scattering} we can assume without loss of generality that $P_{1}:=\gamma(t_{1})$ is not a corner point of $\PD \Omega_{1}$. Using the continuity and injectivity of path $\gamma(t)$, we have $P_{1}\neq P_{0}$ and $\Omega_{1}\cap \{\gamma(t):\ t>t_{1}\}=\phi$. Let $l_{1}\subset \PD \Omega_{1}$ denote the line  segment containing point $P_{1}$ and define $\Omega_{2}:=R_{1}\lb \Omega_{1}\rb$. Using Lemma \ref{lem:3.6} and repeating the previous step, we can define the function $u_{2}^{(2)}(x):=\mathcal{D}_{1}u_{2}^{(1)}(x)$ for $x\in \Omega_{2}\cup D_2^c$. Then we have
		$u_{2}^{(2)}=\mathcal{D}_{1}\mathcal{D}_{0}u_2$
		in $D_2^c$ and $u_{2}^{(2)}(x)=u_{2}^{(1)}(R_1x)$ for $x\in \Omega_2$.
		Hence, it is a solution to the boundary value problem
		\begin{align*}
			\Delta u_{2}^{(2)}+k^{2}u_{2}^{(2)}=0 \quad \mbox{in} \ \Omega_{2},\qquad \mathcal{B}_{2}u_{2}^{(2)}=0\quad \mbox{on}\ \PD \Omega_{2}.
		\end{align*}
		Following the previous argument, we choose a point $P_{2}:=\gamma(t_{2})\neq P_{1}$ for some $t_{2}>t_{1}$ and a line segment $l_{2}\subset \PD \Omega_{2}$ such that $P_{2}\in l_{2}$ and $\Omega_{2}\cap \{\gamma(t):\ t>t_{2}\}=\phi$.
		In general, we can find a polygonal domain $\Omega_{N}:=[R_{N-1}R_{N-2}\cdots R_1](\Omega_1)$, $N\geq 1$ and a function
		\begin{align*}
			u_{2}^{(N)}(x):= [\mathcal{D}_{N-1}\mathcal{D}_{N-2}\cdots \mathcal{D}_{0}u_{2}](x),\quad x\in D_{2}^{c}\cup\Omega_N
		\end{align*}
		such that
		\begin{align}\label{Equation for uN}
			\Delta u_{2}^{(N)}+k^{2}u_{2}^{(N)}=0\quad \mbox{in}\quad \Omega_{N},\qquad \mathcal{B}_{N}u^{(N)}_{2}=0\quad \mbox{on}  \quad \PD \Omega_{N}.
		\end{align}
		To proceed, we suppose that $\{\gamma(t):\ t>t_{N}\}\cap \Omega_{N}=\phi$ for some $t_{N}>t_{N-1}$ and that $P_{N}=\gamma(t_{N})\in l_{N}$ where $l_{N}\subset\PD \Omega_{N}$ is a line segment. Since the path $\gamma(t)$ is connected to infinity in $D_{2}^{c}$, by  \cite[Lemma 3.3]{Elschner_Guang-hui_Inverse_elastic_scattering} we have $t_N\rightarrow\infty$ as $N\rightarrow\infty$.
		After a finite number of steps,
		we have that either $L_{N}$ lies complectly in $D^c_{2}$ or $l_{N}$ together with one of its neighbouring side of $\partial \Omega_N$ extends to two half lines in $D_2^c$. Using the connectness of $D_2$ and Lemma \ref{lem:3.6} (ii), we can extend $u_2^{(N)}$ to the whole space. 
		This implies that $u^{(N-1)}_{2}$,  $u^{(N-2)}_{2}$, $\cdots$,
		$u_{2}^{(0)}=u_2$ can also be extended to $\R^{2}$, which is impossible.
		\hfill$\Box$
		
		\subsection{Proof of Theorem \ref{Th2}}
		If $\lambda=\lambda(x)$ is the restriction of a holomorphic function over $\C^2$ to $\R^2$, the extension formula (\ref{Definition of D0}) for the harmonic equation
		should be replaced by (\ref{eH}) with $\alpha(z,w)\equiv 1$, $\beta=i\lambda(z,w)$ over $\C^2$ and $S(\tau)=\widetilde{S}(\tau)=\tau$ for $\tau\in \C$.
		Existence of the non-local reflection principle for the Helmholtz equation in Theorem \ref{Reflection principle for Helmholtz equation} can be established in the same manner.
		Lemma \ref{Lemma about non-existence of non-parallel lines} carries over to the case that $\eta_j$ ($j=0,1$) are the restriction of holomorphic functions over $\C^2$ to $\R^2$, and Lemma \ref{lem:3.6} still holds true for impedance coefficients that are piecewise holomorphic on the boundary of a polygon. Arguing the same as in the proof of Theorem
		\ref{Main theorem}, we can arrive at the same contradiction if $D_1\neq D_2$.
		
		To prove $\lambda_1(x)=\lambda_2(x)$ on $\partial D$, where $D=D_1=D_2$, we observe that $u_1=u_2=:u$ in $D^c:=\R^2\backslash\overline{D}$ if the corresponding far-field patterns are identical over all observation directions. This yields the coincidence of the Cauchy data of $u_1$ and $u_2$  on $\partial D$. Hence,
		\be\label{eq:impedance}
		0=(\partial_\nu-i\lambda_1(x)) u_1-(\partial_\nu-i\lambda_2(x)) u_2=-i[\lambda_1(x)-\lambda_2(x)]u(x)\quad\mbox{on}\quad \partial D.
		\en
		If $\lambda_1(x_0)\neq\lambda_2(x_0)$ at some $x_0\in \partial D$, by the continuity of $\lambda_j$, there exists a neighborhood of $x_0$ at $\partial D$ such that they are not identical there. Using (\ref{eq:impedance}), we get the vanishing of $u=u_1=u_2$ in an open set of $\partial D$. In view of the impedance boundary condition of $u_j$, we also get the vanishing of the Neumann data on this open set. Now the Holmgren's  theorem gives $u\equiv 0$ in $D^c$ which is impossible in the area far away from $D$.

		\hfill$\Box$
		
		We end up the paper with several remarks.
		\begin{remark}\label{rem}
			Theorem \ref{Th2} holds true even if the impedance coefficient $\lambda(x)$ is piecewise  holomorphic in the sense that the restriction of $\lambda$ to each side of $\partial D$ extends to an entire holomorphic function defined on $\C^2$. In particular, $\lambda(x)$ can be a piecewise-constant function whose values remain the same on each side of $\partial D$.
			For convex polygons, Lemma \ref{Lemma about non-existence of non-parallel lines} and the arguments in the first case of the proof of Theorem
			\ref{Main theorem} are sufficient to imply uniqueness. For non-convex polygons, Lemma \ref{lem:3.6} should be applied to handle the second case.  Further, one can generalize the uniqueness result of Theorem \ref{Th2} to the case of mixed Dirichlet, Neumann and Robin boundary conditions. On each side of the polygon, there should be only one type of these boundary conditions.
		\end{remark}
		
		\begin{remark}
			Our uniqueness proofs to Theorems \ref{Main theorem} and \ref{Th2} imply that the total field cannot be real-analytic around each corner lying on the convex hull of a polygonal obstacle. The proof follows from the same uniqueness argument for identifying convex polygons. The \textquoteleft singularity\textquoteright \  of the total field at corner points might be helpful in designing numerical schemes for imaging a polygonal obstacle (see \cite{Elschner_Guang-hui_Inverse_elastic_scattering,HMY}).
		\end{remark}
		
		\begin{remark}
			The uniqueness results are valid for other form of non-vanishing incoming waves $u^{in}$ that are solutions to the Helmholtz equation in a neighboring area of $D$. In place of using (\ref{eq:plane}) for a plane wave, one needs to consider the Robin boundary value problem
			\ben
			\Delta u^{in}+k^2 u^{in}=0\quad\mbox{in}\quad D,\qquad \partial_\nu u^{in} +i\lambda(x)u^{in}=0\quad\mbox{on}\quad\partial D,
			\enn where the normal is directed into outward. Applying integration by part, it is easy to prove $u^{in}\equiv 0$ in $\R^2$ if $\lambda(x)\geq 0$ on $\partial D$ and  $\lambda(x)\geq \lambda_0>0$ on an open set of $\partial D$.
		\end{remark}
		
		\section{Acknowledgements}
	The authors would like to thank Johannes Elschner and Yubiao Zhang for their comments and suggestions which help improve the original version of this manuscript. M. Vashisth is supported by the NSAF grant (No. U1930402).

	\end{document}